\documentclass[11pt]{amsart}
\usepackage{amscd,amsmath,amssymb,amsfonts,verbatim}
\usepackage[cmtip, all]{xy}
\usepackage{comment}
\usepackage{tikz-cd}

\makeatletter
\@namedef{subjclassname@2020}{%
	\textup{2020} Mathematics Subject Classification}
\makeatother

\newtheorem{thm}{Theorem}[section]

\newtheorem{lem}[thm]{Lemma}
\newtheorem{cor}[thm]{Corollary}

\theoremstyle{definition}

\newtheorem{defn}[thm]{Definition}
\theoremstyle{remark}
\newtheorem{remk}[thm]{Remark}
\newtheorem{remks}[thm]{Remarks}

\newtheorem{exm}[thm]{Example}
\newtheorem{exms}[thm]{Examples}
\newtheorem{notat}[thm]{Notation}
\numberwithin{equation}{section}

{\hfill$\square$\end{defn}}
{\hfill$\square$\end{remk}}
{\hfill$\square$\end{remks}}
{\hfill$\square$\end{exm}}
{\hfill$\square$\end{exms}}
{\hfill$\square$\end{notat}}

\newcommand{\Z}{{\mathbb Z}}

\newcommand{\fm}{{\mathfrak m}}
\newcommand{\fp}{{\mathfrak p}}

\newcommand{\surj}{\twoheadrightarrow}

\newcommand{\by}[1]{\stackrel{#1}{\rightarrow}}

\newcommand{\remove}[1]{}

\newcommand{\iso}{\by \sim}

\newcommand{\rank}{{\rm rank}}

\newcommand{\Hom}{{\rm Hom}}

\newcommand{\Spec}{{\rm Spec \,}}

\newcommand{\hh}{\rm ht}
\newcommand{\Um}{\mbox{\rm Um\,}}

\newcommand{\Aut}{{\operatorname{\rm Aut}}}
\newcommand{\El}{{\operatorname{\rm E}}}

\newcommand{\ds}{{/\kern-3pt/}}

\renewcommand{\dim}{\text{\rm dim}}

\newcommand{\tuborg}{\left\{\begin{array}{ll}}
\newcommand{\sluttuborg}{\end{array}\right.}

\def\ol#1{\overline{#1}}

\newcounter{elno}

\newcounter{elno-abc}   

\newcounter{elno-abc-prime}

\begin{document}

\title[Efficient generation, unimodular element in a geometric ring]{Efficient generation, unimodular element in a geometric subring of a polynomial ring}
\author{Sourjya Banerjee, Chandan Bhaumik and Husney Parvez Sarwar}

\address{(Sourjya Banerjee) Department of Mathematics and Statistics, Indian Institute of Science Education and Research Kolkata, Campus Road, Mohanpur, West Bengal 741246}
\email{sourjya.pdf@iiserkol.ac.in}
\email{sourjya91@gmail.com}

\address{(Chandan Bhaumik) Department  of Mathematics, Indian  Institute of Technology Kharagpur,  Kharagpur 721302, West Bengal, India}
\email{cbhaumik11math@gmail.com}

\address{(H.P. Sarwar) Department  of Mathematics, Indian  Institute of Technology Kharagpur,  Kharagpur 721302, West Bengal, India}
\email{parvez@maths.iitkgp.ac.in}
\email{mathparvez@gmail.com}




\maketitle

\begin{quote}\emph{Abstract.}  
Let $R$ be a commutative Noetherian ring of dimension $d$. First, we define the ``geometric subring" $A$ of a polynomial ring $R[T]$ of dimension $d+1$ (the definition of geometric subring is more general, see (\ref{defn:imp})).
Then we prove that every locally complete intersection ideal of height $d+1$ is a complete intersection ideal. Thus improving the general bound of Mohan Kumar \cite{NMK78} for an arbitrary ring of dimension $d+1$. Afterward, we deduce that
every finitely generated projective $A$-module of rank $d+1$ splits off a free summand of rank one. This improves the general bound of Serre \cite{Serre58} for an arbitrary ring. Finally, applications are given to a set-theoretic generation of an ideal in the geometric ring $A$ and its polynomial extension $A[X]$.

\end{quote}
\setcounter{tocdepth}{1}

\section{Introduction}\label{sec:Intro}
Let $R$ be a commutative Noetherian ring of (Krull) dimension $d$. A classical
result of Serre \cite{Serre58} asserts that every finitely generated projective $R$-module $P$ of rank $d$ splits off a free summand of rank one. In general, it is the best possible result as it is evidenced by the tangent bundle over the real algebraic sphere. Therefore if the $rank(P)\leq d$, then the question of $P$ to splits off a free summand, is subtle and critical. In the case $rank(P)=\dim(A)$, where  $A$ is a smooth affine algebra over an algebraically closed field, there is a
well-developed obstruction theory which is due to Murthy \cite[Theorem 3.7]{Mu94}. One of the
crucial steps of Murthy \cite{Mu94} is the following beautiful theorem of Mohan Kumar \cite{NMK84}.

\begin{thm}
Let $A$ be an affine algebra over an algebraically closed field of dimension $d\ge 2$. Let $P$ be a finitely generated projective $A$-module of rank $d$ and $\phi\in \Hom_A(P,A)$ such that the image ideal $\phi(P)$ has height $d$. Then, $P$ splits off a free summand of rank one if and only if $\phi(P)$ is generated by $d$ elements.
	\end{thm}

Now let us define the efficient generation of an ideal $I$ in a ring $R$
which has an apparent connection with the complete intersection ideal (see Remark \ref{rem:comp}).
The ideal $I$ is called efficiently generated if $\mu(I/I^2)=\mu(I)$, where $\mu(*)$ is the minimum number of generators of $*$ as a $R/I$ (respectively $R$) module. In \cite{NMK77}, Mohan Kumar proved that $\mu(I/I^2)\le \mu(I)\le \mu(I/I^2)+1$. Moreover, in \cite{NMK78}, it was shown by him that $I$ is efficiently generated whenever $\mu(I/I^2)\ge d+1$. It is well known that there exists a real maximal ideal of a real $n$-sphere which is not generated by $n$-many elements. This shows that the previously mentioned bound for efficient generation of an ideal is the best possible in general. Therefore, the question of efficient generation of an ideal $I$ becomes interesting whenever $\mu(I/I^2)=\dim(R)$. In the polynomial ring $R[T]$, due to the results of Sathaye \cite{Sathaye78} (for affine domains over infinite fields) and Mohan Kumar \cite{NMK78}, we know that any ideal $K\subset R[T]$ with the property $\mu(K/K^2)=d+1$ is efficiently generated provided $\hh(K)\ge 1$. This was one of the crucial steps in their solutions of Eisenbud-Evans conjecture on the number of generators of a finite module.

Before we state our results on efficient generation, we offer the following definition.

\begin{defn}\label{defn:imp}
 Let $R$ be a commutative Noetherian ring of dimension $d\geq1$. Let $n\in \Z$ and $f\in R[T]$ be a non-zero divisor.
 A ring $A$ is called a geometric
 subring of $R[T,f^n]$ if
\begin{enumerate}
    \item $A$ is a Noetherian ring of dimension $d+1$ such that $R\subset A \subset R[T,f^n]$;
    \item  there exists a non-zero divisor $s\in R$ such that $A_s=R_s[T,f^n]$. 
\end{enumerate}
 \end{defn}

\begin{remk}
    We call $A$ a geometric subring because our major examples are Rees algebras $R[It]$ and Rees-like algebras $R[It, t^n]$ ($n\in \Z$). These algebras appear naturally in algebraic geometry while blowing up a variety along a subvariety. Another class of examples are  the Noethrian  symbolic Rees algebras.
\end{remk}

For the geometric subring $A$ of $R[T,f^n]$, we have the following result
about the efficient generation of an ideal in $A$ (for a proof, see Theorem \ref{imt}). {\it  Let $A$ be a ring as in (\ref{defn:imp}) (need not be Noetherian). Let $I\subset A$ be an ideal such that $\hh(I)\ge 2$ and $\mu(I)<\infty$. Moreover, assume that $I=<f_0,\ldots,f_d>+I^2$. Then there exist $F_i\in I $ such that $I=<F_0,\ldots,F_d>$ with $F_i-f_i\in I^2$ $(i=0,\ldots, d).$ }

More generally, we prove the following result.

 \begin{thm}(Theorem \ref{thm:main}) \label{thm:main:intro}
 Let $A$ be a ring as in (\ref{defn:imp}). Let $L$ be a finitely generated projective $A$-module of rank one. Let $I\subset A$ be an ideal such that $\hh(I)=\dim(A)$. Further, assume that there is a surjection $\omega_I:(L/IL)\oplus (A/I)^d\surj I/I^2$. Then there exists a surjection $\phi:L\oplus A\surj I$ such that $\phi\otimes A/I=\omega_I$. In particular, the $(d+1)$-th Euler class group $E^{d+1}(A,L)$ is trivial.
 \end{thm}

 

The efficient generation problem was solved for polynomial ring $R[T]$ 
by Mohan Kumar \cite{NMK78}, for Laurent polynomial ring
$R[T,T^{-1}]$ by Mandal
\cite{Ma82}, and for overrings for polynomial ring by Das--Zinna \cite{DZ15}.

As an amusing application, we first derive the following result from Theorem \ref{thm:main:intro}. {\it Every finitely generated projective $A$-module of rank $d+1$ splits off a free summand of rank one (Theorem \ref{eue}).}
In the case of Rees algebras, this result was proved by Rao--Sarwar
\cite{RS19}. In the case of symbolic Rees algebras (need not be Noetherian), it is proved in \cite{BhSa22}.
As we noted in the first paragraph that there is an obstruction to splits off a free summand in the case $\rank=\dim$.  For instance, see \cite{Mu94} and more generally see \cite{BR98},\cite{BR00} for a well-developed obstruction theory
to determine when a finitely generated projective module splits off a free summand of rank one.
However, when the ring $B$ is ``nice", then every finitely generated projective $B$-module $P$ splits off a free summand of rank one even if $\rank(P) \leq \dim(B)$. Examples of such rings are polynomial rings $B=R[X_1,\ldots,X_n]$ (cf. \cite{Pl83}, \cite{BhR84}), Laurent polynomial rings $B=R[X_1,\ldots,X_n,Y_1^{\pm 1},\ldots, Y_m^{\pm 1}]$ (cf. \cite{Ma82}, \cite{BLR85}), monoid rings $R[M]$, where $M$ is a monoid, belongs to a class of commutative cancellative torsion-free monoids (cf. \cite{KeSa17}, \cite{KeMa22}).

An ideal $I$ in a commutative ring $R$ is called set-theoretically generated by $f_1,\ldots,f_n \in R$ if the radical of the ideal $I$, i.e. $\sqrt{I}=\sqrt{<f_1,\ldots,f_n>}=:$ the radical of the ideal generated by $f_1,\ldots, f_n$. In section \ref{4}, we prove that every finitely generated ideal in $A$
(need not be Noetherian) of height $\geq 2$
is set-theoretically generated by $d+1$ elements (Theorem \ref{si}). In section \ref{5}, we extend the results of section \ref{4} in polynomial extension $A[X]$ (Theorem \ref{pe}).

\section*{Acknowledgment}
 H.P. Sarwar acknowledges the grant SRG/2020/000272,
S.E.R.B. Govt. of India. This project was conceived when he was a
 visiting scientist at I.S.I. Kolkata. He would like to thank the Institute for the support.
S. Banerjee would like to thank Professor Mrinal Kanti Das for generously sharing his ideas which help him to solve some of the problems tackled in this article. He also thanks Dr. Md. Ali Zinna for clearing some doubts about set-theocratic complete intersection ideals.

\section{Recollection of basic definitions and results}\label{sec:Prelim}

\subsection{Notation.} We assume that all the rings are commutative with the multiplicative identity $1$ and modules are finitely generated. Also, we assume that the projective modules have a constant rank function. 
Let $R$ be a ring, and let $I,\,J$ be two ideals of $R$. For any element $x\in R$, $<I,x>$ (respectively $<I,J>$) denotes
the ideal generated by $I$ and $x$ 
 (respectively by $I$ and $J$ ).
Let $\sqrt{I}$ denote the radical of the ideal $I$, that is, 
$\sqrt{I}=\{x\in R \,|\, x^n\in I\}$. Let $\hh(I)$ denote the height of the ideal $I$. Let $P$ be a projective $R$-module. Then throughout the paper, $P^*$ denote the $R$-module $\Hom_R(P,R)$.
Let $\Aut(P)$ denote the group of all automorphisms of $P$. 

\smallskip

We begin with the following lemma, which has an origin from a lemma due to Mohan Kumar \cite[Lemma]{NMK77}, slightly recast to suit our needs. For the sake of completeness, we give the proof.

\begin{lem}\label{MKL}
		Let $R$ be a commutative Noetherian ring and $I$ be an ideal of $R$. Let $J, K$ be
		ideals of $A$ contained in $I$ such that $K\subset I^2$ and $I=J+K$. Then there exists $e\in K$ such that $e(1-e)\in J$ and $I=<J,e>.$
\end{lem}

\begin{proof}
    Note that $(I/J)^2=(I^2+J)/J=I/J$ (as $K\subset I^2$ and $I=J+K$). Hence $I/J$ is an idempotent ideal in the Noetherian ring $R/J$. Let `bar' denote going modulo the ideal $J$. Since the image of $K$ maps sujectively onto $I/J$, we get $\ol K\ol I=\ol I$. By Nakayama lemma there exists $e\in K$ such that $(\ol 1-\ol e)\ol I=\ol 0$. Therefore, $(1-e)I=J$, that is, $I+eI=J$. Thus going modulo $e$ we get $I=J$, hence $J+<e>=I$. Since $e\in K\subset I$ and $(1-e)I=J$, we get $e(1-e)\in J$.
\end{proof}

The following result, which is an application of Lemma \ref{MKL}, is required to prove our main theorem.

\begin{lem} \label{rsl}
Let $R$ be a commutative Noetherian ring. Let $I,J\subset R$ be two ideals. Suppose that there exists $e\in I^2$ such that $I=<J,e>$. Moreover, assume that $e\in \sqrt{J}$. Then $J=I$.
\end{lem}

\begin{proof}
     Because of the local-global principle, it is enough to prove that $J_{\fp}=I_{\fp}$ for all $\fp\in \Spec(R)$. Since $J\subset I$, we have $V(I)=\{\fp\in \Spec(R):I\subset \fp\} \subset V(J)$.  Let $\fp \in V(J)$.  Since $e\in \sqrt{J}$, we have $e\in \fp$. Therefore, $\fp\in V(I)$ as $I=<J,e>$. Hence we have $V(I)=V(J)$.

Let $\fp\in \Spec(R)$. If $\fp\not\in V(I)=V(J)$, then $J_{\fp}=I_{\fp}=R_{\fp}.$ Now assume that $\fp\in V(I)$. In the local ring $R_{\fp}$, we have  $I_{\fp}=<J,e>A_{\fp}=J_{\fp}+I^2_{\fp}$ as $e\in I^2$. Using Lemma \ref{MKL}, there exists $s\in I^2_{\fp}$ such that $I_{\fp}= <J,s>R_{\fp}$ and $s(1-s)\in J_{\fp}$. Now since $s\in I^2_{\fp}\subset \fp R_{\fp}$ and $R_{\fp}$ is a local ring, the element $1-s$ is a unit in $R_{\fp}$. Hence we have $s\in J_{\fp}$. Therefore, we get $I_{\fp}=J_{\fp}$. This completes the proof.
\end{proof}

\smallskip
		
		 The following result is a consequence of a theorem by Eisenbud-Evans \cite{EE73}. One can find a proof of the same in \cite[Corollary 2.8]{BR98}.
		 
		\begin{thm}\label{ee}
		Let $R$ be a commutative ring and $P$ be a projective $R$-module of rank $n$. Let $(\alpha,a)\in (P^*\oplus R)$. Then there exists an element $\beta \in P^{*}$ such that $\hh (I_a)\ge  n$,
		where $I=(\alpha+a\beta)(P).$ 
		\end{thm}

The ``free version'' of the next lemma is due to Mohan Kumar \cite[Corollary 3]{NMK78}. 

\begin{lem}
    \label{kl}
	Let $R$ be a commutative Noetherian ring of dimension $d$. Let $L$ be a projective $R$-module of rank one. Let $I\subset R$ be an ideal. Suppose that we are given a surjection $\omega:(L/IL)\oplus (R/I)^d\surj I/I^2$. Then there exists a surjection $\Omega:L\oplus R^d\surj I$ such that $\Omega\otimes R/I=\omega$.  
\end{lem}
\begin{proof}
     Let $\alpha:L\oplus R^d\to I$ be a lift (might not be surjective) of $\omega$. Let $J=\alpha(L\oplus R^d)$. Then note that $I=J+I^2$. Therefore, by Lemma \ref{MKL} there exists $e\in I^2$ such that $I=<J,e>$ and $e(1-e)\in J$. By using Theorem \ref{ee}, there exists $\beta\in (L\oplus R^d)^*$ such that $\hh(J')_e\ge d+1>\dim(R)$, where $J'=\alpha'(L\oplus R^d)$ and $\alpha'=\alpha+e\beta$. Observe that, we have $I=<J,e>=<J',e>$. Since $\dim(R)=d$, we get $J'_e=R_e$. This implies that $e\in \sqrt{J'}$. Therefore, using Lemma \ref{rsl} we get $I=J'$. This finishes the proof of the lemma.
\end{proof}

	\subsection{A brief introduction to the Euler class group}
		
			The purpose of this part is to briefly recall some of the basic definitions and terminologies of the Euler class theory from \cite{BR00}. 
			Before going to the Euler class theory, we recall the following definitions.
			
			\begin{defn}
			Let $R$ be a ring and $P$ a projective $R$-module.
			  An element $p\in P$ is called a unimodular element of $P$ if there exists a $R$-linear map $\phi:P\to R$ such that $\phi(p)=1$. Observe that, $P$ has a unimodular element if and only if $P$ splits off a free summand of rank one.
			\end{defn}
			
			\begin{defn}
			
				Let $R$ be a ring. Let $P$ be a projective $R$-module such that either $P$ or $P^*$ has a unimodular element. We choose $\phi\in P^*(=\Hom_R(P,R))$ and $p \in P$ such that $\phi(p)=0$. We define an endomorphism $\phi_p$ as the composite $\phi_p:P\to R\to P$, where $R\to P$ is the map sending $1\to p.$ Then by a transvection we mean an automorphism of $P$, of the form $1+\phi_p$, where either $\phi\in \Um(P^*)$ or $p\in \Um(P)$. Let $\El(P)$   denote  the subgroup of $\Aut(P)$ generated by all transvections.

			\end{defn}

			 Let $R$ be a commutative Noetherian ring of dimension $d\ge 2$. Let $I\subset R$ be an ideal such that $\mu(I/I^2)=\hh(I)=d$. Let $L$ be a rank one projective $R$-module. Let $\alpha,\beta:(L/IL)\oplus(R/I)^{d-1}\surj I/I^2$ be two surjections. We say $\alpha$ and $\beta$ are related if there
			 exists $\sigma\in \El((L/IL)\oplus (R/I)^{d-1})$ such that $\alpha \sigma = \beta$. This defines an equivalence relation on the set of all surjections $(L/IL)\oplus(R/I)^{d-1}\surj I/I^2$. Let $[\alpha]$ denote the
			 equivalence class of $\alpha$. We will call $[\alpha]$ a ``local orientation" of $I$ with respect to $L$. With abuse of notation sometimes we will call $\alpha$ a local orientation of $I$ with respect to $L$ instead of $[\alpha]$. The local orientation $[\alpha]$  is said to be a ``global orientation" of $I$ with respect to $L$, if there exists a surjective map $\Gamma:L\oplus R^{n-1}\surj I$ such that $\Gamma\otimes R/I=\alpha$. In this situation, we call $\Gamma$ a surjective lift of $\alpha$. Notice that, since the canonical map $\El(L\oplus R^{d-1})\surj \El((L/IL)\oplus (R/I)^{d-1})$ is surjective, if $\alpha$ has a surjective lift, then so has any $\beta$ equivalent to $\alpha$. Whenever $L\cong R$, we will say a local orientation (respectively global orientation) of $I$ instead of a local orientation (respectively global orientation) of $I$ with respect to $R$.
			

		 Let $G$ be the free Abelian group on the set $B$ of pairs $(I, 
		 \omega_J )$, where:
		\begin{enumerate}
			\item $I \subset R$ is an ideal such that $\mu(I/I^2)=\hh(I)=d$;
			\item  $\Spec(R/I)$ is connected;
			\item $\omega_I : (L/IL)\oplus (R/I)^{d-1}\surj I/I^2$
			is an equivalence class of local orientations of $I$ with respect to $L$.
		\end{enumerate}

	 Let $H$ be the subgroup of $G$ generated by the set $S$ of pairs $(J, \omega_J )$, where $\omega_J$ is an equivalence class of global orientations of $J$ with respect to $L$. Then the quotient group $G/H$ is the ``$d$-th Euler class group" of $R$ with respect to $L$, denoted as $E^d(R,L)$. Whenever $L\cong R$, we will write $E^d(R)$ instead of $E^d(R,R)$.

		 Let $P$ be a projective $R$-module of rank $d$ with determinant $L$. Let $\chi:\wedge^d(L\oplus R^{d-1})\iso \wedge^d (P)$ be an isomorphism. Let $\lambda : P \surj I$ be a surjection, where $I$ is an ideal of $R$ of height
			$d$. Let `bar' denote going modulo $I$. We obtain an induced surjection $\lambda\otimes (R/I) : P/IP \surj  I/I^2$. Note that, since $P$ has determinant $L$ and
			$\dim (R/I) =0$, we have $P/IP\cong (L/IL)\oplus (R/I)^{d-1}$. We choose an isomorphism $\phi :(L/IL)\oplus (R/I)^{d-1}\iso  P/IP$, such that $\wedge^d\phi=\chi\otimes R/I$. Let $\omega_I$ be the surjection $ (\lambda\otimes R/I)\circ\phi: (L/IL)\oplus(R/I)^{d-1}\surj I/I^2$. We say that $(I, \omega_I)$ is an ``Euler cycle'' induced by the triplet
			$(P,\lambda,\chi)$. 
			
			Whenever the Euler cycle is independent of the choice of $\lambda$, that is, if $(J,\omega_J)$ is another Euler cycle induced by a triplet $(P,\lambda',\chi)$, for some $\lambda'\in P^*$, then  $(I,\omega_I)=(J,\omega_J)$ in $E^d(R,L)$, in this situation we shall call the class of $(I,\omega_I)$ in the Euler class group $E^d(R,L)$ as the ``Euler class" of the pair $(P,\chi)$, denoted as $e(P,\chi)$. It was proved in \cite{BR00}, that whenever the ring contains an infinite field and either $R$ is smooth or $\frac{1}{d!}\in R$, then the Euler class of such a pair $(P,\chi)$ is well defined.

The following theorem is due to Bhatwadekar-Sridharan \cite[Theorem 4.2]{BR00}. This has been used crucially and frequently in the course of proving our main theorems.
			
			\begin{thm}
   \label{tcecg}
			Let $R$ be a commutative Noetherian ring of dimension $d\ge 2$. Let $L$ be a projective $R$-module of rank one. Let $(I,\omega_I)\in E^d(R,L)$. Then $(I,\omega_I)=0$ in $E^d(R,L)$ if and only if $\omega_I$ is a global orientation of $I$ with respect to $L$.  
			\end{thm}

    

\section{Efficient generation and triviality of the Euler class group}

The following lemma is a modification of a result due to Bhatwadekar-Sridharan \cite[Lemma 5.6]{BR00}, and it is crucial to our main theorem. For the proof, we shall follow their arguments with some necessary alterations.

		\begin{lem}\label{brcl}
		Let $R$ be a commutative Noetherian ring of dimension $d\ge 2$. Let $L$ be a projective $R$-module of rank one. Let $(I,\omega_I)\in E^d(R,L)$ such that $(I,\omega_I)\neq 0$ in $E^d(R,L)$. Further, assume that there exists a non-zero divisor $s\in R$ such that the following hold:
		
		\begin{enumerate}
			\item $\dim(R_s)=d$;
			\item $\hh(I_s)=d$;
			\item $(I,\omega_I)_s=0$ in $E^d(R_s, L_s)$, where $(I,\omega_I)_s :=(I_s,\omega_I\otimes (R/I)_s)$.
		\end{enumerate} 
		Then there exists $(J,\omega_J)\in E^d(R,L)$ such that $(I,\omega_I)=(J,\omega_J)$ in $E^d(R,L)$ and $s\in \sqrt{J}$.
		\end{lem}
		
	\begin{proof}
	 
		We define $B:=\frac{R}{sR\cap I^2}$. Let $`bar'$ denote going modulo the ideal $sR\cap I^2$. Since $s$ is a non-zero divisor, we have  $\dim(B)\le d-1$. 
  Therefore, in the ring $B$, we have 
  $$\ol{\omega_I}:(\ol{L}/\ol{I}\ol{L})\oplus (B^{d-1})\surj \ol{I}/\ol{ I}^2.$$
  
  Using Lemma \ref{kl}, we get a surjective map $\ol{\Gamma_I}:\ol L\oplus B^{d-1}\surj \ol I$ such that $\ol{\Gamma_I}\otimes B/\ol{I}=\ol \omega_I$. Let $\Gamma_I:L\oplus R^{d-1}\to I$ be a lift of $\ol{\Gamma_I}$ which a priory need not be surjective. Then we get $I=J_{\Gamma_I}+I^2\cap sR$, where $J_{\Gamma_I}=\Gamma_I(L\oplus R^{d-1})$. Using Lemma \ref{MKL}, there exits $e\in I^2\cap sR$ such that $I=<J_{\Gamma_I},e>$, and $e(1-e)\in J_{\Gamma_I}$. Moreover, using Theorem \ref{ee} replacing ${\Gamma_I}$ by ${\Gamma_I+e\Gamma_I'}$, for some suitably chosen $\Gamma_I'\in (L\oplus R^{d-1})^*$, we have $\hh(({J_{\Gamma_I}})_e)\ge d$, that is, ${(J_{\Gamma_I}})_e=R_e$ or $\hh(({J_{\Gamma_I}})_e)= d$.  If ${(J_{\Gamma_I}})_e=R_e$, then $e\in \sqrt{{(J_{\Gamma_I}})}$. Hence using Lemma \ref{rsl}, we have $J_{\Gamma_I}=I$. This gives us $\Gamma_I:L\oplus R^{d-1}\surj I$ is a surjective lift of $\omega_I$. Hence using Theorem \ref{tcecg} we have $(I,\omega_I)=0$ in $E^d(R,L)$ which is a contradiction.
  
  Therefore, we can assume that $\hh((J_{\Gamma_I})_e)=d$. Let $I_1=<J_{\Gamma_I},1-e>$. Then note that we have the following:
		
		\begin{enumerate}
			\item[(a)] $\Gamma_I:L\oplus R^{d-1} \surj I\cap I_1=J_{\Gamma_I}$ is a surjection;
			\item[(b)]  $I_1+eR=I_1+sR=I_1+I=R$;
			\item[(c)] $\hh(I_1)=d$.
		\end{enumerate} 
	
	 Since $I_1+I=R$, using Chinese remainder theorem,
  we have $R/(I\cap I_1)\cong R/I \oplus R/I_1$. Hence
  $\Gamma_I$ induces local orientations of $I$ and $I_1$ with respect to $L$. Since $\Gamma_I$ is a lift of $\omega_I$, any local orientation of $I$ with respect to $L$ induced by  $\Gamma_I$, matches with $(I,\omega_I)$ in $E^d(R,L).$ Let $\omega_1'$ be a local orientation of $I_1$ with respect to $L$ induced by $\Gamma_I$. Then by (a) we have 
  \begin{equation}\label{eqn-1}
      (I,\omega_I)+(I_1,\omega_1')=0.
  \end{equation}
	 
	  As $I_1+sR=R$ and $\hh(I_1)=d$, we get $\hh((I_1)_s)=d$. These give us the class of $ ({I_1},\omega_1')_s$ is in $E^d(R_s,L_s)$. Since we have $(I,\omega_I)_s=0$ in $E^d(R_s,L_s)$, from \ref{eqn-1}, we have $({I_1},\omega_1')_s=0$ in $E^d(R_s,L_s)$. Therefore, using Theorem \ref{tcecg}, we obtain a surjective lift $$\tau':L_s\oplus {R_s}^{d-1}\surj (I_1)_s$$
   of $\omega_1'\otimes (R/I_1)_s$. 
   
   Since $L$ is finitely generated, there exists an integer $k\ge 1$ such that $(s^{2k}\tau')(L_s\oplus {R_s}^{d-1})\subset I_1$. Let $\tau =(s^{2k}\tau')\circ i:L\oplus R^{d-1}\to I_1$, where $i:L\oplus R^{d-1}\to L_s\oplus {R_s}^{d-1}$ is the canonical map induced by the localization map. Since $s$ is a unit modulo $I_1$, the map $\tau\otimes R/I_1:(L/I_1L)\oplus (R/I_1)^{d-1}\surj I_1/I_1^2$ is surjective. That is, $\tau\otimes R/I_1$ is a local orientation of $I_1$ with respect to $L$. Let $\omega_1''=\tau\otimes R/I_1$. By the construction of $\tau$ we have $(I_1,\omega_1')=(I_1,u^2\omega_1'')$, for some $u\in (R/I_1)^*$ (here $u$ will be some power of $s$). Using \cite[Lemma 5.4]{BR00} we get $(I_1,u^2\omega_1'')=(I_1,\omega_1'')$ this implies that $(I_1,\omega_1')=(I_1,\omega_1'')$. Therefore, from \ref{eqn-1} we get
\begin{equation}\label{eqn-2}
    (I,\omega_1)+(I_1,\omega_1'')=0.
\end{equation}

	  Now adapting the proof of \cite[Lemma 5.6, paragraph 2]{BR00}, we get $(J, \omega_J)\in E^d(R,L)$ such that the followings hold:
	  
	  \begin{enumerate}
	  	\item $(J)_s=R_s$;
	  	\item $I_1+J=R$;
	  	\item $(I_1,\omega_1'')+(J, \omega_J)=0$ in $E^d(R,L)$.
	  \end{enumerate}

	Since $(J)_s=R_s$ we have $s\in \sqrt{J}$. Now combining (\ref{eqn-2}) and $(3)$, we have $(I,\omega_I)=(J,\omega_J)$ in $E^d(R,L)$. This finishes the proof. 	   
	\end{proof}		


 \begin{thm}\label{thm:main}
 Let $A$ be a ring as in (\ref{defn:imp}). Let $L$ be a finitely generated projective $A$-module of rank one. Let $I\subset A$ be an ideal such that $\hh(I)=\dim(A)=d+1$. Further, assume that there is a surjection $\omega_I:(L/IL)\oplus (A/I)^d\surj I/I^2$. Then there exists a surjection $\phi:L\oplus A\surj I$ such that $\phi\otimes A/I=\omega_I$. In particular, 
 the $(d+1)$-th Euler class group $E^{d+1}(A,L)$ is trivial.
 \end{thm}
 
 \begin{proof}
      Let $S$ be the multiplicatively closed subset of $R$ consisting all non-zero divisors. Then it is easy to check that $S^{-1}(R)$ is an Artinian ring and 
      $S^{-1}A\cong S^{-1}R[T,f^n]$. To show that
      $S^{-1}L$ is a free module, we can assume that $S^{-1}R[T,f^n]$ is reduced. But then $S^{-1}R$
      is a product of fields, hence $S^{-1}R[T,f^n]$
      is a product of  principal ideal domains. 
      So we conclude that $S^{-1}L$ is a free $S^{-1}A$-module of rank one.
      
      Since $L$ is finitely generated, there exists $s_1\in S$ such that $L_{s_1}$ is a free $A_{s_1}$-module of rank one.  Since $A$ is a geometric ring, there exists a non-zero divisor $s_2\in R$ such that $A_{s_2}=R_{s_2}[T,f^n]$. Let $s=s_1s_2$. Then note that $s\in R$ is a non-zero divisor, such that $L_s\cong A_s$ and  $A_s= R_s[T,f^n]$. Now we prove the theorem in the following two steps.
		
		\textbf{Step 1.} In this step, we will show that we can assume that $s\in \sqrt I$.

		{\bf Proof of the Step 1.}  Note that $\dim(A_s)\le \dim(A)=d+1$. Therefore either $\dim(A_s)=d+1$ or $\le d$. We shall handle these two cases separately.
  
\smallskip		

	{\bf Case 1.}	Let us assume that $\dim(A_s)=d+1$. 
 
\smallskip

 Recall that there is a one-to-one correspondence between the prime ideals of $A_s$ with the prime ideals of the ring $A$, not containing $s$. Therefore, for any prime ideal $\fp\subset A$, $s\not\in \fp$ implies that $\hh(\fp_s)=\hh(\fp)$. Let $\sqrt{I}=\bigcap_{i=1}^n\fp_i $, where $\fp_i$'s are prime ideals of the ring $A$. Since $\hh (I)=\hh (\sqrt{I})=d+1$ and $\dim(A)=d+1$, we have $\hh(\fp_i)=d+1$ for all $i=1,\ldots,n$. Hence we have either  $\hh(I_s)=d+1$ or $I_s=A_s$. Now if $I_s=A_s$ then $s\in \sqrt{I}$. So we have already achieved the claim of Step 1.
 
 Therefore, we assume that $\hh(I_s)=d+1$. 	Observe that, $A_s$ is an overring of the polynomial ring $R[T]$. We have $\hh(I)=d+1$, and $\hh(I_s)= d+1\ge 2$. Therefore, using \cite[Theorem 3.2 ]{DZ15}, we have $(I_s,\omega_I\otimes (A/I)_s)=0$ in $E^{d+1}(A_s)$. Note that if $(I,\omega_I)=0$ in $E^{d+1}(A,L)$, then using Theorem \ref{tcecg} we get $\omega_I$ is a global orientation of $I$ with respect to $L$. Thus the theorem is proved. Therefore, without loss of generality, we may assume that $(I,\omega_I)\not=0$ in $E^{d+1}(A,L)$. Hence by the previous reductions we may assume that the triplet $(A,(I,\omega_I), s)$ satisfies all the hypotheses of Lemma \ref{brcl}. Applying Lemma \ref{brcl}, we get $(J,\omega_J)\in E^{d+1}(A,L)$ such that $(I,\omega_I)=(J,\omega_J)$ and $s\in \sqrt{J}$. Note that to prove the theorem in view of Theorem \ref{tcecg}, it is enough to show that $(J,\omega_J)=0$ in $E^{d+1}(A,L)$. Therefore, replacing $(I,\omega_I)$ by $(J,\omega_J)$ we may assume that $s\in \sqrt{I}$. This completes the proof of Step 1 in the case whenever $\dim(A_s)=d+1$.

\smallskip
    
	{\bf Case 2.}	Suppose that $\dim(A_s)\le d$.

 \smallskip
 
 Let $B=A/sA\cap I^2$. Let `bar' denote going modulo the ideal $sA\cap I^2$. Since $\dim(A_s)<d+1$, by Lemma \ref{kl}, $\ol{\omega_I}$ is a global orientation of $\ol I$ with respect to $\ol L$. Let $\Gamma:L\oplus A^d\to I$ be a lift of $\ol {\omega_I}$. Then we have $I=K+I^2\cap sA$, where $K=\Gamma(L\oplus A^d)$. Using Lemma \ref{MKL}, there exists $e\in I^2\cap <s>$ such that $I=<K,e>$, and $e(1-e)\in K$. Moreover, using Theorem \ref{ee}, we can make the assumption $\hh (K'_e)\ge d+1$, where $K'=(\Gamma+e\lambda)(L\oplus A^d)$, for some suitably chosen $\lambda\in (L\oplus A^d)^*$. As we have $I=<K',e>$ and $e\in <s>$, these imply that $\hh(K'_s)\ge \hh(K'_e)\ge d+1$. Since $\dim(A_s)<d+1$, we get $s\in \sqrt{K'}\subset \sqrt{I}$. This finishes the proof of Step 1.
		
	\smallskip
		
		\textbf{Step 2.} In this step, we show that $\omega_I$ is a global orientation of $I$ with respect to $L$.

\smallskip
  
		{\bf Proof of Step 2.} Note that by Theorem \ref{tcecg} it is enough to show that $(I,\omega_I)=0$ in $E^{d+1}(A,L)$. Since $s$ is a non-zero divisor, by Moving lemma \cite[Corollary 2.14]{BR00} we get an ideal $I'\subset A$ of height $d+1$ such that $I'+I=I'+<s>=A$ and a surjection $\beta :L\oplus A^d\surj I\cap I'$, with $\beta\otimes A/I=\omega_I$. Since $I'+<s>=A$ the map $\omega_1:=\beta\otimes A_{1+sA}: L_{1+sA}\oplus  (A_{1+sA})^d\surj I_{1+sA}\cap I'_{1+sA}= I_{1+sA} $ is surjective. Observe that $\omega_1\otimes (A_{1+sA}/I_{1+sA})=(\beta\otimes A/I)\otimes (A_{1+sA}/I_{1+sA})=\omega_I \otimes (A_{1+sA}/I_{1+sA})$. 
		
		In the ring $A_s$ we have $I_s=A_s$, as $s\in \sqrt I$. Let $\omega_2:(A_s)^{d+1}\surj I_s$, which sends $e_1$ to $1$ and $e_i$ to $0$ for all $i>1$. Since $I_s=A_s$, the equality $\omega_2\otimes (A_{s}/I_{s})=\omega_I\otimes A_{s}$ holds trivially. Consider the following fiber product diagram:
		$$\begin{tikzcd}
			L\oplus A^{d} \ar[rrr] \ar[dd]\ar[dr,twoheadrightarrow, dashed] &&&  (A^{d+1})_s \ar[dd]\ar[dr, twoheadrightarrow ] \\
			& I \ar[rrr,crossing over] \ar[dd]  &&& I_s\ar[dd, twoheadrightarrow] \\
			(L\oplus A^{d})_{1+sA} \ar[rr] \ar[dr,twoheadrightarrow] & &  (L\oplus A^{d})_{s(1+sA)} \ar[dr,twoheadrightarrow] \ar[r,"\sim"] & (A^{d+1})_{s(1+sA)}\ar[dr,twoheadrightarrow] \\
			& I_{1+sA}\ar[rr] && I_{s(1+sA)} \ar[r] & I_{s(1+sA)}\ar[from=uu,crossing over]
		\end{tikzcd}$$
		Note that $ A_{s(1+sA)}=(R_s[T,f^n])_{1+sA}=S'^{-1}R'[T]$, where $R'=R_{s(1+sR)}$ and $S'$ is some multiplicatively closed subset of non-zero divisors of $R'[T]$. Therefore, $A_{s(1+sA)}$ is an overring of the polynomial ring $R'[T]$. Since $\dim(R')\le d-1$, we have $\dim(A_{s(1+sA)})\le d$. Also observe that since $L_{s(1+sA)}\cong A_{s(1+sA)}$, we get $(\omega_1)_s, (\omega_2)_{1+sA}\in \Um_{d+1}( A_{s(1+sA)})$. Hence by a theorem of Rao \cite[Theorem 5.1 (I)]{Rao82}, there exists $\epsilon\in E_{d+1}(A_{s(1+sA)})$ such that $((\omega_2)_{1+sA})\epsilon=(\omega_1)_s$. Since elementary matrices are homotopic to identity, using Quillen's splitting theorem \cite[Lemma 1]{Quillen76} the matrix $\epsilon$ splits. Using the universal property of fiber product, we get a unique map $\phi:L\oplus A^{d}\to I$ such that $\phi_{1+sA}=\omega_1$ and $\phi_s=\omega_2$. To prove that $\phi$ is a surjective lift of $\omega_I$ note that it is enough to check locally. Let $\fm\subset A$ be a maximal ideal. Then $\fm$ is either co-maximal with $s$ or co-maximal with $1+s$. Therefore, in any case, in the ring $A_{\fm}$, the map $\phi$ matches with either of $\omega_1$ or $\omega_2$. This completes the proof. 
 \end{proof}

\begin{remk}\label{non:Noe}
 Taking $L= A$ in Theorem \ref{thm:main} we get for any ideal $I\subset A$ if $\mu(I/I^2)=\dim(A)=\hh(I)$ then $\mu(I)=\dim(A)$. Moreover, it says that any set of $d+1$ generators of $I/I^2$ lifts to a set of $d+1$ generators of $I$. Here we would like to mention that one can weaken the hypothesis that $A$ is a Noetherian ring by only assuming $I$ is finitely generated. This can be done in the following way. Let $I=<f_1,\ldots,f_k>$. Since $A$ is geometric subring of $R[T,f^n]$, there exists $g_1,g_2\in A$ such that $\frac{g_1}{s^{m_1}}=T$ and $\frac{g_2}{s^{m_2}}=f^n$ in $R_s[T,f^n]$. Now consider $B=R[f_1,\ldots,f_n,g_1,g_2]$. Then $B$ is finitely generated $R-$algebra such that $I\subset B$ and $B_s=A_s=R_s[T,f^n]$. Now it is enough to show that $I$ is efficiently generated in $B$.
\end{remk}

 \begin{remk} \label{rem:comp}
     Let $R$ be a commutative Noetherian ring of dimension $d$ and $I$ an ideal of $R$ with $\hh(I)=d$. If
     $\mu(I/I^2)=d$, then it is easy to check that $I$ is
     a local complete intersection ideal of height $d$. Hence Theorem \ref{thm:main} says that every local complete intersection ideal of height $d+1$ in $A$ is a complete intersection ideal of height $d+1$ in $A$.
 \end{remk}

 Continuing with the same notations as in Theorem \ref{thm:main}, in the next theorem we show that any projective $A$-module $P$ of rank $d+1$ splits off a free summand of rank one. Note that here we are assuming the ring is Noetherian. For Rees algebras or extended Rees algebras this has been proved in \cite{RS19}.
		
		\begin{thm}	
         \label{eue}
		Let $A$ be as in Theorem \ref{thm:main}.
  Let $P$ be a projective $A$-module such that $\rank(P)=d+1$. Then $P$ has a unimodular element.    
		\end{thm}

   \begin{proof}
We deduce the theorem from Theorem \ref{thm:main} and Subtraction principle \cite[Corollary 3.4]{BR00}.
  For this,  let $L:= \wedge^{d+1} P$ and $Q:= L\oplus A^{d-1}$.
  Then $\wedge^{d}Q=L$.
  Let $\chi:\wedge^{d+1} P\cong \wedge^{d+1}(Q\oplus A)(\cong L)$ be an isomorphism. Let $(I,\omega_I)\in E^{d+1}(A,L)$ be a Euler cycle induced by the trilet $(P,\lambda,\chi)$ for some suitably chosen $\lambda: P\surj I$. This means that there exists an isomorphism $\phi :(L/IL)\oplus (A/I)^{d} \iso  P/IP$, such that $\wedge^{d+1}\phi=\chi\otimes A/I$ and a surjection $\omega_I:=(\lambda\otimes A/I)\circ\phi: (L/IL)\oplus(A/I)^{d}\surj I/I^2$. 
  
  By Theorem \ref{thm:main}
  we have a surjective lift $\beta:{L\oplus A^{d}}\surj I$ such that $\beta \otimes A/I=\omega_I$. That is, $\beta \otimes A/I=(\lambda\otimes A/I)\circ\phi$. Now from \cite[Corollary 3.4]{BR00}, it follows that $P$ has a unimodular element. 
   \end{proof}

In the next theorem, we relax some hypothesis on the height of the ideal of Theorem \ref{thm:main}.

\begin{thm} \label{imt}
			Let $A$ be as in (\ref{defn:imp}).
   Let $I\subset A$ be an ideal such that $\hh(I)\ge 2$ and $\mu(I/I^2)\le d+1$. Then $\mu(I)\le d+1$. Moreover, any set of $d+1$ generators of $I=<a_0,\ldots,a_d>+I^2$ lifts to a set of $d+1$ generators of $I$.
\end{thm}
	
	\begin{proof}
 Note that if $\mu(I/I^2)<d+1$, then by Lemma \ref{MKL} we get $\mu(I)\le d+1$. Hence we can take $\mu(I/I^2)=d+1$. By Lemma \ref{MKL} there exists $e\in I^2$ such that $I=<a_0,\ldots,a_d,e>$, and $e(1-e)\in <a_0,\ldots,a_d>$. Moreover, using Theorem \ref{ee}, we can make the assumption that $\hh(<b_0,\ldots,b_d>_e)\ge d+1$, where $b_i=a_i+e\lambda_i$ for all $i=0,\ldots,d$. Note that $a_i-b_i\in I^2$, for all $i=0,\ldots,d$. Let $J=<b_0,\ldots,b_d,1-e>$. Since $J+<e>=A$ and $\hh(<b_0,\ldots,b_d>_e)\ge d+1$ we get $\hh(J)\ge d+1$. Now if $\hh(J)>d+1$, then the result follows from the fact that $I\cap J=<b_0,\ldots,b_d>$. Therefore, we may assume that $\hh(J)=d+1$.
	
	Observe that $J=<b_0,\ldots,b_d,1-e>$ induces $J=<b_0,\ldots,b_d>+J^2$. Now using Theorem \ref{thm:main}, there exist $c_i\in J$ $(i=0,\ldots,d)$ such that $J=<c_0,\ldots,c_d>$ with $c_i-b_i\in J^2$, for $i=0,\ldots,d$. Applying Subtraction principle \cite[Proposition 2.2]{DaSr10}, we can find $f_i\in I$ such that $I=<f_0,\ldots,f_d>$ with $f_i-b_i\in I^2$, for $i=0,\ldots,d$. This completes the proof.    
	\end{proof}

\begin{remk}
Following the argument given in Remark \ref{non:Noe}, Theorem \ref{imt} is also valid for non-Noetherian geometric subring $A$ of $R[T]$, with an extra hypothesis $\mu(I)<\infty.$
\end{remk}

 \section{Application 1: Set-theoretic generation of ideals in $A$}\label{4}
		
Let $A$ be as in (\ref{defn:imp}).
As an application of our main theorem, in this section, we will deduce some results on the set-theoretic generation of ideals in $A$. 

	\smallskip
	

The proof of the following theorem can be found in \cite[Theorem]{Bo80}.
 
\begin{thm} \label{Kt}
	 Let $R$ be a commutative Noetherian ring of dimension $d$. Let $I \subset R$ be an ideal. 
 Then there exists an ideal $J$ satisfying $\mu(J/J^2
	) \le d$ and $\sqrt{I}=\sqrt{J}$.   
\end{thm}

We are now ready to prove our main result in this section.

\begin{thm}\label{si}
Let $A$ be a ring as in (\ref{defn:imp}) but need not be Noetherian.
Let $I\subset A$ be an ideal such that $\hh(I)\ge 2$ and $\mu(I)<\infty$. Then $I$ is set-theoretically generated by at most $d+1$ elements.
\end{thm}

\begin{proof}
  Observe that following the arguments given in Remark \ref{non:Noe}, it is enough to assume that $A$ is a Noetherian ring. It follows from Theorem \ref{Kt} that there exists an ideal  $J\subset A$ such that $\mu(J/J^2)\le d+1$ and $\sqrt{I}=\sqrt{J}$. Therefore it is enough to show that $\mu(J)\le d+1$. The remaining part of the proof is devoted to showing this.

Note that if $\mu(J/J^2)\le d$ then by Lemma \ref{MKL} we get $\mu(J)\le d+1$. 
So we can assume that $\mu(J/J^2)=d+1$.
Since $\sqrt{J}=\sqrt{I}$, we have $\hh(J)=\hh(I)\ge 2$. Hence we can use Theorem \ref{imt} to achieve that $\mu(J)\le d+1$. This completes the proof. 
\end{proof}

\section{Application 2: Set-theoretic generation of ideals in $A[X]$}\label{5}
		
 Let $A$ be as in (\ref{defn:imp}).
 As an application of our main theorem, in this section, we will deduce some results on the set-theoretic generation of ideals in $A[X]$.

\begin{thm}
\label{pe}
Let $R$ be a commutative Noetherian ring of dimension $d\ge 1$ which satisfies one of the following conditions:

\begin{enumerate}
	\item $\mathbb{Q}\subset R$;
	\item $R$ is an affine algebra over $\overline{\mathbb{F}}_p$.
\end{enumerate}
	Let $A$ be as in (\ref{defn:imp}) but need not be Noetherian, where $R$ in (\ref{defn:imp}) is one of the above two.
 Let $I\subset A[X]$ be an ideal of height $\ge 3$ such that $\mu(I)<\infty$. Moreover, assume that $I=<f_1,\ldots,f_{d+1}>+I^2$. Then there exists $g_i\in I$ such that $I=<g_1,\ldots,g_{d+1}>$ with $f_i-g_i\in I^2$ for all $i=1,\ldots,d+1$.  
\end{thm}

\begin{proof}
We divide the proof into the following two cases.

\textbf{Case 1.} In this case we assume that $\mathbb{Q}\subset A$.  Observe that following the arguments given in Remark \ref{non:Noe}, without loss of generality we may assume that $A$ is a Noetherian ring containing $\mathbb{Q}$. Since $\mathbb{Q}\subset A$, using \cite[Lemma 3.3]{BR98}, there exists $\lambda\in \mathbb{Q}$ such that $I(\lambda)=A$ or $\hh(I(\lambda))\ge 3$. Moreover, via the automorphism $T\to T-\lambda$, we may assume that either $I(0)=A$ or $\hh(I(0))\ge 3$. Note that if $\hh(I(0))\ge 3$, then by Theorem \ref{imt}, there exists $b_i\in I(0)$ such that $I(0)=<b_1,\ldots,b_{d+1}>$, where $f_i(0)-b_i\in I(0)^2$ for $i=1,\ldots,d+1$. Using \cite[Remark 3.9]{BR98} there exists $F_i\in I$ such that $I=<F_1,\ldots,F_{d+1}>+I^2X$, where $F_i-f_i\in I^2$ and $F_i(0)=b_i$ for $i=1,\ldots, d+1$. Observe that, to prove the theorem in this case, it is enough to find a lift of generators of $I=<F_1,\ldots,F_{d+1}>+I^2X$. The remaining part of the proof, in this case, is devoted to showing this.

Let $S$ be the multiplicatively closed set consisting of all monic polynomials in $R[X]$. Let $B=S^{-1}R[X]$. Let $A(X)$ be the ring inverting all monic polynomials in $A[X]$. In view of \cite[Theorem 4.2]{DaSr10}, it is enough to show that the set of generators $IA(X)=<F_1,\ldots,F_{d+1}>A(X)+I^2A(X)$ can be lifted to a set of generators of $IA(X)$. Moreover, since $A(X)$ is a localization of $S^{-1}(A[X])$, it is enough to lift $S^{-1}I=S^{-1}<F_1,\ldots,F_{d+1}>+S^{-1}(I^2)$ to a set of generators of $S^{-1}I$. Observe that we have the following:
\begin{enumerate}
	\item $\dim(B)=\dim(R)=d$;
	\item $B\subset S^{-1}A[X]\subset B[T,f^n]$;
	\item $(S^{-1}A[X])_s=S^{-1}(A_s[X])=B_s[T,f^n]$;
	\item $S^{-1}I\subset S^{-1}A[X]$ is an ideal such that $\hh(S^{-1}I)\ge 3$.
\end{enumerate}

In particular, $S^{-1}A[X]$ is a geometric subring of $ B[T,f^n]$. Therefore, applying Theorem \ref{imt} we can lift $S^{-1}I=S^{-1}<F_1,\ldots,F_{d+1}>+S^{-1}I^2$ to a set of generators of $S^{-1}I$. This completes the proof in Case 1.

\textbf{Case 2.}  In this case we assume that $R$ is an affine algebra over $\overline{\mathbb{F}}_p$. Again as before, following the arguments given in Remark \ref{non:Noe}, without loss of generality we may assume that $A$ is a finitely generated $R-$algebra. In particular, this will imply $A$ is an affine $\overline{\mathbb{F}}_p-$algebra. Therefore, to prove the theorem, in view of \cite[Theorem 9.1]{SBMKD2} it is enough to find a lift of $IA(X)=<f_1,...,f_{d+1}>A(X)+I^2A(X)$ to a set of generators of $IA(X)$. This can be achieved following the same arguments given in Case  1.
\end{proof}
 
 As customary to the Euler class theoretic literature, we get the following result as a corollary to the previous theorem.
 \begin{cor}
 \label{pe1}
 Let $R$ and $A$ be as in Theorem \ref{pe}. Moreover, assume that $A$ is finitely generated $R$-algebra. Let $P$ be a projective $A[X]$-module of rank $d+1$ with a trivial determinant. Then $P$ has a unimodular element.
\end{cor}

 \begin{proof}
 From Theorem \ref{pe} we get $E^{d+1}(A[X])=0$. Therefore, the result follows from \cite[Corollary 4.11]{Das03} and \cite[Theorem 4.11]{SBMKD2}.
 \end{proof}

 \begin{remk}
   Whenever $\mathbb{Q}\subset R$, altering the proof of Theorem \ref{pe} one can establish the fact that $E^{d+1}(A[X],L)$ is trivial, where $L$ is a projective $A[X]-$module of rank one. Therefore, using \cite{DZ13}, one can remove the hypothesis that $P$ has a trivial determinant from Corollary \ref{pe1}. 
 \end{remk} 
 
 We now move towards proving a similar result to Theorem \ref{si}, in the polynomial extensions. As a preparation, we need the following theorem due to Ferrand-Szpiro \cite{F-S} (or see \cite[Theorem 6.1.3, page no 75]{Ma_Book_97}).
 
 \begin{thm}
 \label{fs}
Let $R$ be a commutative Noetherian ring. Let $I$ be a locally complete intersection ideal of $R$ of height $n \ge 2$ with $\dim(R/I) \le 1$. Then there is a locally complete intersection ideal $J \subset R$ of height $n$ such that 
\begin{enumerate}
	\item $\sqrt{I}=\sqrt{J}$;
	\item $J/J^2$ is a free $R/J-$module of rank $n$.
\end{enumerate}     
 \end{thm}
 
 \begin{thm}
  Let $R$ and $A$ be as in Theorem \ref{pe}. Let $I\subset A[X] $ be a local complete intersection ideal of height $\ge 3$ such that $\dim(A[X]/I)\le 1$ and $\mu(I)<\infty$. Then $I$ is  set theoretically generated by $d +1$ or lesser elements.
 \end{thm}

 \begin{proof}
Using Remark \ref{non:Noe}, without loss of generality we may assume that $A$ is finitely generated $R$-algebra. Let $\hh(I) = h$. It follows from the Theorem \ref{fs} that there exists an ideal $J \subset A[X]$ such that the followings hold:
\begin{enumerate}
	\item $\sqrt{I}=\sqrt{J}$;
	\item $J/J^2$
	is a free $A[X]/J$-module of rank $h$.
\end{enumerate}
 
 If $h \le d$, then from Lemma \ref{MKL} $J$ is generated by at-most $d + 1$ elements. Therefore, we may assume that $h = d + 1$. Let
 $J = <f_1,..., f_{d+1}> + J^2$. Applying Theorem \ref{pe}, we see that $J$ is generated by $d + 1$
 elements. Therefore, from $(1)$, it follows that $I$ is set theoretically generated by $d+1$ elements. This completes the proof.  
 \end{proof}


\begin{thebibliography}{99}





\bibitem{SBMKD2}
	S. Banerjee and M.K. Das.
	{\sl Splitting criteria of projective modules on polynomial extensions over various base rings.} Pre-print June 2022. https://arxiv.org/abs/2206.06819


\bibitem{BLR85}{} S.M. Bhatwadekar, H. Lindel and R.A. Rao, {\sl The
  Bass-Murthy question: Serre dimension of Laurent polynomial
  extensions}, Invent. Math. {\bf 81} (1985) 189-203.
  
 
  


	
	\bibitem{BR98}
	S.M. Bhatwadekar and R.Sridharan.
	{\sl Projective generation of curves in polynomial extensions of an affine
	domain and a question of {N}ori.}
	 { Inventiones Mathematicae}, {\bf 133(1):} 161--192, jun 1998.
	


\bibitem{BR00} S. M. Bhatwadekar, R. Sridharan; {\sl  The Euler Class Group of a Noetherian Ring.} Compositio
Mathematica {\bf 122(2)} (2000), 183 - 222. \


\bibitem{BhR84}{} S.M. Bhatwadekar and A. Roy, {\sl Some theorems about projective modules over polynomial rings}, J. Algebra {\bf 86}
  (1984) 150-158.

  \bibitem{BhSa22} C. Bhaumik, H.P. Sarwar,
  {\sl Existence of unimodulear element in a projective module over a symbolic Rees algebra}, Pre-print, 2022. 
https://doi.org/10.48550/arXiv.2301.05394



\bibitem{Bo80} M. Boraty\'nski, {\sl Generating ideals up to radical.} Arch. Math. (Basel), {\bf 33}, (1979/80), no. 5, 423--425.






	
	\bibitem{Das03}
	M.K. Das.
	{\sl The {E}uler class group of a polynomial algebra.}
	 { Journal of Algebra}, {\bf 264(2)}, 582--612, jun 2003.
	
	\bibitem{DaSr10}
	M.K. Das and R. Sridharan.
	{\sl Good invariants for bad ideals.}
	 { Journal of Algebra}, {\bf 323(12):} 3216--3229, 2010.
	
\bibitem{DZ13}
	M.K. Das and M.A. Zinna.
	{\sl The {E}uler class group of a polynomial algebra with coefficients in
	a line bundle.}
	{ Mathematische Zeitschrift}, {\bf 276(3-4)}, 757--783, 2013.
	

	\bibitem{DZ15}
	M.K. Das and M.A. Zinna.
	{\sl Efficient generation of ideals in overrings of polynomial rings.\/}
	\newblock {Journal of Pure and Applied Algebra}, {\bf 219(9):} 4016--4034, 2015.
	



\bibitem{EE73}
	D. Eisenbud and E.G. Evans, Jr.
	{\sl Generating modules efficiently: theorems from algebraic {$K$}-theory\/}
 { Journal of Algebra}, {\bf 27:} 278--305, 1973.














\bibitem{KeSa17} M.K. Keshari, H.P. Sarwar, 
{\sl Serre Dimension of Monoid Algebras\/}, 
Proc. Math. Sci. Indian Acad. Sci., {\bf 127} (2017), no. 2, 269–280.\


\bibitem{KeMa22}
M.K. Keshari and M.A. Mathew. {\sl On Serre dimension of monoid algebras and Segre extensions}, Journal of Pure and Applied Algebra {\bf 226} (2022), no. 9, Paper No. 107058, 16 pp.
 








\bibitem{NMK77}
	N. Mohan Kumar.
	{\sl Complete intersections.}
	 {Journal of Mathematics of Kyoto University}, {\bf 17(3)}, 
    533--538, 1977.


\bibitem{NMK78}
	N. Mohan Kumar.
	{\sl On two conjectures about polynomial rings.}
	 { Inventiones Mathematicae}, {\bf 46(3)}, 225--236, 1978.
	
	\bibitem{NMK84}
	N. Mohan Kumar.
	{\sl Some theorems on generation of ideals in affine algebras.}
	 { Commentarii Mathematici Helvetici}, {\bf 59(1)}, 243--252, 1984.
	
	
	







  \bibitem{Ma82}
	S. Mandal.
	{\sl Basic elements and cancellation over {L}aurent polynomial rings.} {Journal of Algebra}, {\bf 79(2)}, 251--257, 1982.
	

\bibitem{Ma_Book_97}
	S.~Mandal.
	{\sl Projective modules and complete intersections}.
	 Springer Berlin Heidelberg, Oct. 1997.

	\bibitem{Mu94}
	M.P. Murthy.
	{\sl Zero cycles and projective modules.}
	 {The Annals of Mathematics}, {\bf 140(2)}, 405, 1994.
	
	\bibitem{Pl83}
	B. Plumstead.
	{\sl The conjectures of {E}isenbud and {E}vans.}
	{American Journal of Mathematics}, {\bf 105(6)}, 1417--1433, 1983.







\bibitem{Quillen76}{} D. Quillen. {\sl Projective modules over polynomial rings,}
 Invent. Math. {\bf 36} (1976), 167--171.  
 
 
\bibitem{Rao82} R.A. Rao, {\sl Stability theorems for overrings of polynomial rings. II.}
J. Algebra {\bf 78} (1982), no. 2, 437--444.

\bibitem{RS19}
	R.A. Rao and H.P. Sarwar.
	{\sl Stability results for projective modules over {R}ees algebras.}
 Journal of Pure and Applied Algebra, {\bf 223(1):} 1--9, 2019.
	



\bibitem{Sathaye78}
	A. Sathaye.
	{\sl On the {F}orster-{E}isenbud-{E}vans conjectures.}
	 {Inventiones Mathematicae}, {\bf 46(3)}, 211--224, 1978.

  
\bibitem{Serre58}{} J.-P. Serre, {\sl Modules projectifs et espaces
  fibr\'es \`a fibre vectorielle}, 1958 S\'eminaire P. Dubreil,
  M.-L. Dubreil-Jacotin et C. Pisot, 1957/58, Fasc. 2, Expos\'e 23 18
  pp.
	



\bibitem{F-S}
	L. Szpiro.
	 {\sl Lectures on equations defining space curves}.
  Tata Institute of Fundamental Research, Bombay; Springer-Verlag,
	Berlin-New York, 1979. Notes by N. Mohan Kumar.





\end{thebibliography}
\end{document}